\documentclass[letterpaper, 10 pt, conference]{ieeeconf}  

\IEEEoverridecommandlockouts                              

\usepackage{amsmath} 
\usepackage{amssymb}  

\newtheorem{thm}{Theorem}

\newtheorem{defn}{Definition}
\newtheorem{rem}{Remark}

\usepackage{comment}

\title{\LARGE \bf
Application of Symmetry Groups to the Observability Analysis of Partial Differential Equations
}

\author{Bernd Kolar, Hubert Rams, and Markus Sch{\"o}berl
\thanks{This work has been supported by the Austrian Science Fund (FWF) under grant number P 29964-N32. All authors are with the Institute of Automatic Control and Control Systems Technology, Johannes Kepler University Linz, Altenbergerstrasse 66, 4040 Linz, Austria.\newline {Email: \tt\small bernd.kolar@jku.at}}
}

\begin{document}

\maketitle
\thispagestyle{empty}
\pagestyle{empty}

\begin{abstract}

Symmetry groups of PDEs allow to transform solutions continuously into other solutions.
In this paper, we use this property for the observability analysis of nonlinear PDEs with input and output. 
Based on a differential-geometric representation of the nonlinear system, we derive conditions for the existence of special symmetry groups that do not change the trajectories of the input and the output. If such a symmetry group exists, every solution can be transformed into other solutions with the same input and output trajectories but different initial conditions, and this property can be used to prove that the system is not observable. We also put emphasis on showing how the approach simplifies for linear systems, and how it is related to the well-known observability concepts from infinite-dimensional linear systems theory.

\end{abstract}



\setlength{\arraycolsep}{2pt} 

\section{Introduction}

Symmetry groups of (nonlinear) partial differential equations (PDEs)
were introduced by S. Lie in the late nineteenth century, and have
a wide range of applications. Roughly speaking, a symmetry group of
a system of PDEs is a group which transforms solutions of the system
into other solutions, see \cite{Olver1993}. Thus, symmetry groups
can be used to construct new solutions from given ones. In this contribution,
we shall employ symmetry groups to analyze the observability \textendash{}
or rather, to prove the non-observability \textendash{} of nonlinear
infinite-dimensional systems with input and output.

The mathematical framework for the calculation of symmetry groups
is differential geometry. Even though they are probably not as widely
used as functional-analytic approaches, differential-geometric methods
have turned out to be well-suited for the system- and control-theoretic
analysis of PDEs, see e.g. \cite{Pommaret1994}, \cite{Stormark2000},
\cite{Krasilshchik1986}, \cite{GulliverLasiecka2004}, \cite{SchaftMaschke2002},
\cite{SchoeberlSiukaAutomatica}, or \cite{SchoeberlCovariant2007},
to mention but a few.

The observability problem is about determining the initial conditions
of a system uniquely from the trajectories of the input and the output.
Following the terminology used in \cite{HermannKrener} and \cite{NijmeijerVdS}
for finite-dimensional systems, a pair of initial conditions is said
to be indistinguishable if for every admissible trajectory of the
input, the system generates for both initial conditions the same trajectory
of the output. In other words, the initial conditions are said to
be indistinguishable if they determine the same input-output map.
The system is said to be observable, if (locally) there exists no
pair of indistinguishable initial conditions. As pointed out in \cite{HermannKrener},
already for finite-dimensional nonlinear systems the choice of the
input trajectory is important, since the observability of a system
does not imply that every input trajectory distinguishes two initial
conditions. For linear systems the situation is simpler. Because of
the superposition principle the choice of the input does not matter:
If one input distinguishes two initial conditions, then every input
does.

In \cite{RiegerMTNS2010} and \cite{RiegerCDC2007}, symmetry groups
have already been used to show that a system of nonlinear PDEs with
input and output is not ``observable along a trajectory''. Furthermore,
in \cite{RiegerSchlacherSchoeberlIFAC2008} they have also been used
to study the accessibility of nonlinear PDEs with input. The term
``along a trajectory'' in \cite{RiegerMTNS2010} and \cite{RiegerCDC2007}
means that the observability problem is considered only for a fixed
choice of the input trajectory. Substituting the trajectory of the
input into the system equations yields an autonomous, time-variant
system, and then symmetry groups are used to show that there exist
indistinguishable initial conditions that generate the same output
trajectory.

In this contribution, in contrast, we consider the full observability
problem, where the input is free, for a class of second-order nonlinear
PDEs with a single input and a single output. The idea is very simple
and roughly speaking as follows: If there exists a symmetry group
that does not change the trajectories of the input and the output,
then the system cannot be observable. This is due to the fact that
such a symmetry group allows to transform every solution into (infinitely
many) other solutions with the same input trajectory and the same
output trajectory, but different initial conditions. These initial
conditions are indistinguishable, and consequently the system is not
observable. We also put emphasis on showing how the symmetry group
approach simplifies for linear systems. In particular, we want to
point out how it is related to the well-known observability concepts
from infinite-dimensional linear systems theory, that can be found
e.g. in \cite{CurtainZwart1995}. Of course, it is important to remark
that such a comparison suffers from the different solution concepts
for PDEs. For the calculation of symmetry groups we need a differential-geometric
framework and consider like in \cite{Olver1993} only smooth solutions,
whereas the semigroup theory used in \cite{CurtainZwart1995} is based
on mild or generalized solutions.

The paper is structured as follows: First, in Section \ref{sec:GeometricRepresentation}
we discuss the representation of the considered class of PDEs as submanifolds
of certain jet manifolds. This differential-geometric framework is
the basis for the calculation of symmetry groups, which is discussed
in Section \ref{sec:SymmetryGroups}. In Section \ref{sec:ObservabilityAnalysis}
we show how symmetry groups can be used for our control-theoretic
application, and demonstrate it by means of two examples. Finally,
in Section \ref{sec:LinearSystems} we show how our approach simplifies
for linear systems.

\section{\label{sec:GeometricRepresentation}Geometric Representation of PDEs}

In this contribution, we consider nonlinear PDEs
\begin{equation}
\partial_{t}x^{\alpha}(z,t)=f^{\alpha}(z,t,x(z,t),\partial_{z}x(z,t),\partial_{z}^{2}x(z,t),u(t))\,,\label{eq:Nonlinear_PDE}
\end{equation}
$\alpha=1,\ldots,n_{x}$, on a 1-dimensional spatial domain $\Omega=(0,1)\subset\mathbb{R}$
with a single input $u(t)$, boundary conditions 
\begin{equation}
\begin{array}{ccl}
g^{\lambda}(t,x(0,t),\partial_{z}x(0,t)) & = & 0\,,\quad\lambda=1,\ldots,n_{A}\\
h^{\mu}(t,x(1,t),\partial_{z}x(1,t)) & = & 0\,,\quad\mu=1,\ldots,n_{B}\,,
\end{array}\label{eq:Nonlinear_PDE_BC}
\end{equation}
and an output function 
\begin{equation}
y(t)=c(t,x(z_{0},t),\partial_{z}x(z_{0},t))\label{eq:Nonlinear_PDE_y}
\end{equation}
defined at some point $z_{0}\in\bar{\Omega}$. The functions $f^{\alpha}$,
$g^{\lambda}$, $h^{\mu}$, and $c$ are assumed to be smooth, and
$n_{A}$ and $n_{B}$ denote the number of boundary conditions at
$z=0$ and $z=1$. As usual, by $\bar{\Omega}=[0,1]$ we denote the
closure of $\Omega$. Throughout this paper, we take for granted that
solutions of the PDEs (\ref{eq:Nonlinear_PDE}) with the boundary
conditions (\ref{eq:Nonlinear_PDE_BC}) exist and are uniquely determined
by the initial condition $x(z,0)$ and the input function $u(t)$
(well-posedness, see e.g. \cite{Renardy}). Since our focus is on
analyzing the PDEs from a formal geometric point of view, we shall
not verify this assumption. This is in general a difficult task, which
would require additional functional-analytic methods. 

It should be noted that we consider an input $u(t)$ that only depends
on the time $t$ and not on the spatial variable $z$, even though
it acts on the domain $\Omega$ and not on the boundary. The motivation
for this restriction is that in many engineering applications we do
not have an input $u(z,t)$ that can be chosen as a function of $z$
and $t$, but rather an input $u(t)$ that appears in the PDEs multiplied
with some fixed function of $z$, i.e. in the form $b(z)u(t)$.

In the following, we discuss the representation of the considered
nonlinear systems as submanifolds of certain jet manifolds. This differential-geometric
framework is the basis for the calculation of symmetry groups, see
\cite{Olver1993}. For an introduction to differential geometry and
to jet bundles we refer e.g. to \cite{Boothby}, \cite{Spivak}, \cite{Saunders},
and \cite{Giachetta}. We frequently use index notation and especially
the Einstein summation convention to keep formulas short and readable.
Thus, we write e.g. a vector field on an $m$-dimensional manifold
$\mathcal{M}$ with coordinates $x=(x^{1},\ldots,x^{m})$ as $v=v^{\alpha}(x)\partial_{x^{\alpha}}$
instead of $v=\sum_{\alpha=1}^{m}v^{\alpha}(x)\partial_{x^{\alpha}}$.
The Lie derivative of a function $\varphi(x)$ along a vector field
$v$ is denoted by $L_{v}(\varphi)$. To avoid mathematical subtleties,
we assume that all functions, vector fields, etc., are smooth. Furthermore,
it is important to emphasize that all our investigations are only
local. 

For a differential-geometric representation of the PDEs (\ref{eq:Nonlinear_PDE}),
we introduce the bundle $(\mathcal{E},\pi,\bar{\Omega}\times\mathbb{R}^{+})$,
where $\mathcal{E}$ is a $(3+n_{x})$-dimensional manifold with coordinates
$(z,t,x,u)$, $\bar{\Omega}\times\mathbb{R}^{+}$ is a $2$-dimensional
space-time manifold with coordinates $(z,t)$, and $\pi$ is the canonical
projection given in coordinates by $\pi:(z,t,x,u)\rightarrow(z,t)$.
The second jet manifold $J^{2}(\mathcal{E})$ has coordinates $(z,t,x,u,x_{z},x_{t},u_{z},u_{t},x_{zz},x_{zt},x_{tt},u_{zz},u_{zt},u_{tt})$,
i.e. the coordinates of $\mathcal{E}$ plus the derivatives of $x$
and $u$ with respect to $z$ and $t$ up to order two.\footnote{Note that $x$ is here an abbreviation for $(x^{1},\ldots,x^{n_{x}})$.
Likewise, $x_{z}$ is an abbreviation for $(x_{z}^{1},\ldots,x_{z}^{n_{x}})$,
and so on.} In this framework, the PDEs (\ref{eq:Nonlinear_PDE}) can be represented
as a subvariety $\mathcal{S}^{2}\subset J^{2}(\mathcal{E})$, which
is determined by the equations 
\begin{equation}
\begin{array}{rcl}
x_{t}^{\alpha}-f^{\alpha}(z,t,x,x_{z},x_{zz},u) & = & 0\,,\quad\alpha=1,\ldots,n_{x}\\
u_{z} & = & 0\\
u_{zz} & = & 0\\
u_{zt} & = & 0\,.
\end{array}\label{eq:Nonlinear_PDE_geometric}
\end{equation}
Here the additional equations for $u_{z}$, $u_{zz}$, and $u_{zt}$
incorporate the restriction that we only allow solutions where $u$
does not depend on $z$. To avoid mathematical subtleties, we assume
that the Jacobian matrix of the $n_{x}+3$ functions 
\[
\begin{array}{l}
x_{t}^{\alpha}-f^{\alpha}(z,t,x,x_{z},x_{zz},u)\,,\quad\alpha=1,\ldots,n_{x}\\
u_{z}\\
u_{zz}\\
u_{zt}
\end{array}
\]
with respect to the coordinates of $J^{2}(\mathcal{E})$ has maximal
rank $n_{x}+3$ on the subvariety $\mathcal{S}^{2}$. With this assumption,
$\mathcal{S}^{2}$ is a regular submanifold of $J^{2}(\mathcal{E})$
(of codimension $n_{x}+3$).\footnote{The superscript in $\mathcal{S}^{2}$ highlights that it is a submanifold
of the second jet manifold.}

The boundary conditions (\ref{eq:Nonlinear_PDE_BC}) are equations
on manifolds $\mathcal{B}_{A}$ and $\mathcal{B}_{B}$ with coordinates
$(t,x,u,x_{z},x_{t},u_{z},u_{t})$, i.e. with all coordinates of $J^{1}(\mathcal{E})$
except for $z$. Provided that the Jacobian matrices of the functions
$g^{\lambda}(t,x,x_{z})$, $\lambda=1,\ldots,n_{A}$ and $h^{\mu}(t,x,x_{z})$,
$\mu=1,\ldots,n_{B}$ have both maximal rank $n_{A}$ and $n_{B}$,
the boundary conditions describe regular submanifolds $\mathcal{S}_{A}^{1}\subset\mathcal{B}_{A}$
and $\mathcal{S}_{B}^{1}\subset\mathcal{B}_{B}$.

Within this paper, we consider only smooth solutions of PDEs. A smooth
section $\gamma:\bar{\Omega}\times\mathbb{R}^{+}\rightarrow\mathcal{E}$
\[
\begin{array}{ccccccc}
z & = & z & \quad & x^{\alpha} & = & \gamma_{x}^{\alpha}(z,t)\\
t & = & t &  & u & = & \gamma_{u}(z,t)
\end{array}
\]
of the bundle $(\mathcal{E},\pi,\bar{\Omega}\times\mathbb{R}^{+})$
is a solution of the PDEs (\ref{eq:Nonlinear_PDE}) with the boundary
conditions (\ref{eq:Nonlinear_PDE_BC}) if and only if its second
prolongation $j^{2}(\gamma):\bar{\Omega}\times\mathbb{R}^{+}\rightarrow J^{2}(\mathcal{E})$,
given in coordinates by 
\begin{equation}
\begin{array}{cclcccl}
z & = & z & \quad & x^{\alpha} & = & \gamma_{x}^{\alpha}(z,t)\\
t & = & t &  & u & = & \gamma_{u}(z,t)\\
\\
x_{z}^{\alpha} & = & \partial_{z}\gamma_{x}^{\alpha}(z,t) & \quad & x_{zz}^{\alpha} & = & \partial_{z}^{2}\gamma_{x}^{\alpha}(z,t)\\
x_{t}^{\alpha} & = & \partial_{t}\gamma_{x}^{\alpha}(z,t) &  & x_{zt}^{\alpha} & = & \partial_{t}\partial_{z}\gamma_{x}^{\alpha}(z,t)\\
u_{z} & = & \partial_{z}\gamma_{u}(z,t) &  & x_{tt}^{\alpha} & = & \partial_{t}^{2}\gamma_{x}^{\alpha}(z,t)\\
u_{t} & = & \partial_{t}\gamma_{u}(z,t) &  & u_{zz} & = & \partial_{z}^{2}\gamma_{u}(z,t)\\
 &  &  &  & u_{zt} & = & \partial_{t}\partial_{z}\gamma_{u}(z,t)\\
 &  &  &  & u_{tt} & = & \partial_{t}^{2}\gamma_{u}(z,t)\,,
\end{array}\label{eq:Nonlinear_PDE_section}
\end{equation}
satisfies 
\begin{equation}
\begin{array}{rcc}
\left(x_{t}^{\alpha}-f^{\alpha}(z,t,x,x_{z},x_{zz},u)\right)\circ j^{2}(\gamma) & = & 0\\
u_{z}\circ j^{2}(\gamma) & = & 0\\
u_{zz}\circ j^{2}(\gamma) & = & 0\\
u_{zt}\circ j^{2}(\gamma) & = & 0
\end{array}\label{eq:Nonlinear_PDE_solution}
\end{equation}
on $\Omega\times\mathbb{R}^{+}$, as well as 
\begin{equation}
\begin{array}{ccc}
\left.g^{\lambda}(t,x,x_{z})\circ j^{1}(\gamma)\right|_{z=0} & = & 0\\
\left.h^{\mu}(t,x,x_{z})\circ j^{1}(\gamma)\right|_{z=1} & = & 0
\end{array}\label{eq:Nonlinear_PDE_BC_solution}
\end{equation}
on $0\times\mathbb{R}^{+}$ and $1\times\mathbb{R}^{+}$, respectively.
Because of the last three equations in (\ref{eq:Nonlinear_PDE_solution}),
a section (\ref{eq:Nonlinear_PDE_section}) can only be a solution
if $\gamma_{u}$ is independent of $z$.

The condition (\ref{eq:Nonlinear_PDE_solution}) is equivalent to
the statement that the image of $\Omega\times\mathbb{R}^{+}$ under
the map $j^{2}(\gamma)$, written as $j^{2}(\gamma)(\Omega\times\mathbb{R}^{+})$,
must lie entirely in the submanifold $\mathcal{S}^{2}\subset J^{2}(\mathcal{E})$
determined by the equations (\ref{eq:Nonlinear_PDE_geometric}), see
\cite{Olver1993}. Likewise, condition (\ref{eq:Nonlinear_PDE_BC_solution})
is equivalent to the statement that the images of the boundaries $0\times\mathbb{R}^{+}$
and $1\times\mathbb{R}^{+}$ under the restricted maps $\left.j^{1}(\gamma)\right|_{z=0}$
and $\left.j^{1}(\gamma)\right|_{z=1}$, written as $\left.j^{1}(\gamma)\right|_{z=0}(0\times\mathbb{R}^{+})$
and $\left.j^{1}(\gamma)\right|_{z=1}(1\times\mathbb{R}^{+})$, must
lie entirely in the submanifolds $\mathcal{S}_{A}^{1}\subset\mathcal{B}_{A}$
and $\mathcal{S}_{B}^{1}\subset\mathcal{B}_{B}$ determined by the
boundary conditions.%

\section{\label{sec:SymmetryGroups}Symmetry Groups}

For an extensive introduction to Lie groups, transformation groups,
and symmetry groups of differential equations, we refer to \cite{Olver1993}.
In the following, we briefly recapitulate some basics. First, a Lie
group is a group that carries the structure of a smooth manifold,
so the group elements can be continuously varied. More precisely,
an $r$-parameter Lie group carries the structure of an $r$-dimensional
manifold in such a way that both the group operation and the inversion
are smooth maps between manifolds. Second, a transformation group
acting on some manifold $\mathcal{M}$ is a Lie group $G$ together
with a map from (an open subset of) $G\times\mathcal{M}$ to $\mathcal{M}$
that satisfies certain properties. Thus, to each group element $g\in G$
there is associated a map from $\mathcal{M}$ to itself, and this
map is a diffeomorphism on $\mathcal{M}$ (where it is defined). An
important example of a (1-parameter) transformation group is the flow
$\Phi_{\varepsilon}$ of a vector field $v$ defined on $\mathcal{M}$.
Here the Lie group is an interval $I_{0}\subset\mathbb{R}$ containing
0, and for every $\varepsilon\in I_{0}$, $\Phi_{\varepsilon}$ is
a diffeomorphism on $\mathcal{M}$. We shall use this type of transformation
group throughout the paper. Finally, a symmetry group of a system
of PDEs (\ref{eq:Nonlinear_PDE}) is, roughly speaking, a transformation
group acting on the space of independent and dependent variables $\mathcal{E}$
that maps solutions onto solutions. The following definition of a
symmetry group can be found in \cite{Olver1993}.
\begin{defn}
\label{def:SymmetryGroup}%
A symmetry group of the system (\ref{eq:Nonlinear_PDE}) is a local
group of transformations $G$ acting on an open subset of the space
of independent and dependent variables $\mathcal{E}$, with the property
that whenever $\gamma$ is a solution of (\ref{eq:Nonlinear_PDE}),
and whenever $g\cdot\gamma$ with $g\in G$ is defined, then $g\cdot\gamma$
is also a solution of (\ref{eq:Nonlinear_PDE}). Here $g\cdot\gamma$
denotes the application of the diffeomorphism associated with the
group element $g$ to the solution $\gamma$. 
\end{defn}

Instead of considering arbitrary transformation groups acting on $\mathcal{E}$,
for our control-theoretic application we make two simplifications.
First, we consider only transformation groups that do not affect the
independent variables. With respect to the bundle $(\mathcal{E},\pi,\bar{\Omega}\times\mathbb{R}^{+})$,
this means that the transformations shift points of $\mathcal{E}$
only in vertical direction, i.e., tangent to the fibers. For this
reason, we also speak of vertical transformation groups. Second, for
proving that a system is not observable, it is sufficient to consider
only 1-parameter transformation groups, where the group elements can
be varied by a single group parameter. We denote such a vertical 1-parameter
transformation group that acts on $\mathcal{E}$ by $\Phi_{\varepsilon}$,
with the group parameter $\varepsilon$. In coordinates, it is given
by 
\begin{equation}
\Phi_{\varepsilon}:(z,t,x,u)\rightarrow(z,t,\Phi_{x,\varepsilon}(z,t,x,u),\Phi_{u,\varepsilon}(z,t,x,u))\,.\label{eq:VerticalTransformationGroup}
\end{equation}
For every $\varepsilon$ in some interval $I_{0}\subset\mathbb{R}$
containing zero, (\ref{eq:VerticalTransformationGroup}) is a diffeomorphism
on $\mathcal{E}$, and for $\varepsilon=0$ it is the identity map.

Every vertical 1-parameter transformation group $\Phi_{\varepsilon}$
that acts on $\mathcal{E}$ is generated by a vector field
\begin{equation}
v=v_{x}^{\alpha}(z,t,x,u)\partial_{x^{\alpha}}+v_{u}(z,t,x,u)\partial_{u}\label{eq:VerticalTransformationGroup_InfGen}
\end{equation}
on $\mathcal{E}$. This vector field is called the infinitesimal generator,
and can be calculated from the coordinate representation (\ref{eq:VerticalTransformationGroup})
of $\Phi_{\varepsilon}$ via the relation 
\begin{equation}
v=\left(\left.\partial_{\varepsilon}\Phi_{x,\varepsilon}^{\alpha}\right|_{\varepsilon=0}\right)\partial_{x^{\alpha}}+\left(\left.\partial_{\varepsilon}\Phi_{u,\varepsilon}\right|_{\varepsilon=0}\right)\partial_{u}\,.\label{eq:InfGenerator}
\end{equation}
Since we consider a transformation group (\ref{eq:VerticalTransformationGroup})
that does not affect the independent variables, the infinitesimal
generator is a vertical vector field, which means that it is tangent
to the fibers of the bundle $(\mathcal{E},\pi,\bar{\Omega}\times\mathbb{R}^{+})$.
The transformation group $\Phi_{\varepsilon}$ is just the flow of
this vector field, with the flow parameter $\varepsilon$. This one-to-one
correspondence between 1-parameter transformation groups and their
infinitesimal generators is very useful for the calculation of symmetry
groups. The conditions, which a transformation group $\Phi_{\varepsilon}$
must satisfy to be a symmetry group of a system of PDEs, can be formulated
in terms of its infinitesimal generator $v$. Since the transformation
group operates on $\mathcal{E}$ but the PDEs (\ref{eq:Nonlinear_PDE})
determine (algebraic) equations on $J^{2}(\mathcal{E})$, these conditions
involve the second prolongation
\[
\begin{array}{ccl}
j^{2}(v) & = & v_{x}^{\alpha}\partial_{x^{\alpha}}+v_{u}\partial_{u}+\\
 &  & +d_{z}(v_{x}^{\alpha})\partial_{x_{z}^{\alpha}}+d_{t}(v_{x}^{\alpha})\partial_{x_{t}^{\alpha}}+\\
 &  & +d_{z}(v_{u})\partial_{u_{z}}+d_{t}(v_{u})\partial_{u_{t}}+\\
 &  & +d_{zz}(v_{x}^{\alpha})\partial_{x_{zz}^{\alpha}}+d_{zt}(v_{x}^{\alpha})\partial_{x_{zt}^{\alpha}}+d_{tt}(v_{x}^{\alpha})\partial_{x_{tt}^{\alpha}}+\\
 &  & +d_{zz}(v_{u})\partial_{u_{zz}}+d_{zt}(v_{u})\partial_{u_{zt}}+d_{tt}(v_{u})\partial_{u_{tt}}
\end{array}
\]
of $v$. Here
\[
d_{z}=\partial_{z}+x_{z}^{\alpha}\partial_{x^{\alpha}}+u_{z}\partial_{u}+x_{zz}^{\alpha}\partial_{x_{z}^{\alpha}}+x_{zt}^{\alpha}\partial_{x_{t}^{\alpha}}+u_{zz}\partial_{u_{z}}+u_{zt}\partial_{u_{t}}
\]
and
\[
d_{t}=\partial_{t}+x_{t}^{\alpha}\partial_{x^{\alpha}}+u_{t}\partial_{u}+x_{zt}^{\alpha}\partial_{x_{z}^{\alpha}}+x_{tt}^{\alpha}\partial_{x_{t}^{\alpha}}+u_{zt}\partial_{u_{z}}+u_{tt}\partial_{u_{t}}
\]
are the total derivatives with respect to $z$ and $t$. For repeated
total derivatives of a function $\varphi(z,t,x,u)$ we use the abbreviations
$d_{zz}(\varphi)=d_{z}(d_{z}(\varphi))$, $d_{zt}(\varphi)=d_{t}(d_{z}(\varphi))$,
and $d_{tt}(\varphi)=d_{t}(d_{t}(\varphi))$.

The vector field $j^{2}(v)$ is defined on $J^{2}(\mathcal{E})$,
and it is the infinitesimal generator of the second prolongation $j^{2}(\Phi_{\varepsilon}):J^{2}(\mathcal{E})\rightarrow J^{2}(\mathcal{E})$
of $\Phi_{\varepsilon}$, which is a transformation group on $J^{2}(\mathcal{E})$.
The coordinate representation of $j^{2}(\Phi_{\varepsilon})$ can
be obtained from (\ref{eq:VerticalTransformationGroup}) by adding
all first and second total derivatives of $\Phi_{x,\varepsilon}(z,t,x,u)$
and $\Phi_{u,\varepsilon}(z,t,x,u)$ with respect to $z$ and $t$.%

The following theorem provides conditions which ensure that a vertical
vector field generates a 1-parameter symmetry group of the system
(\ref{eq:Nonlinear_PDE}) with boundary conditions (\ref{eq:Nonlinear_PDE_BC}).
It should be noted that in \cite{Olver1993} only the case without
boundary conditions is considered. Therefore, we need additional conditions,
which ensure that the transformation group does not violate the boundary
conditions. 
\begin{thm}
\label{thm:SymmetryGroup}If the prolongations of a smooth vector
field (\ref{eq:VerticalTransformationGroup_InfGen}) satisfy the conditions
\begin{equation}
\begin{array}{rcc}
L_{j^{2}(v)}\left(x_{t}^{\alpha}-f^{\alpha}(z,t,x,x_{z},x_{zz},u)\right) & = & 0\\
L_{j^{2}(v)}u_{z} & = & 0\\
L_{j^{2}(v)}u_{zz} & = & 0\\
L_{j^{2}(v)}u_{zt} & = & 0
\end{array}\label{eq:SymmetryGroup_Cond_PDE}
\end{equation}
on the submanifold $\mathcal{S}^{2}\subset J^{2}(\mathcal{E})$, and
the conditions 
\begin{equation}
\begin{array}{ccl}
\left.L_{j^{1}(v)}g^{\lambda}(t,x,x_{z})\right|_{z=0} & = & 0\\
\left.L_{j^{1}(v)}h^{\mu}(t,x,x_{z})\right|_{z=1} & = & 0
\end{array}\label{eq:SymmetryGroup_Cond_BC}
\end{equation}
on the submanifolds $\mathcal{S}_{A}^{1}\subset\mathcal{B}_{A}$ and
$\mathcal{S}_{B}^{1}\subset\mathcal{B}_{B}$, then it is the infinitesimal
generator of a vertical 1-parameter symmetry group of the system (\ref{eq:Nonlinear_PDE})
with the boundary conditions (\ref{eq:Nonlinear_PDE_BC}). 
\end{thm}

\begin{proof}
The condition (\ref{eq:SymmetryGroup_Cond_PDE}) ensures that $v$
is the infinitesimal generator of a symmetry group of the system (\ref{eq:Nonlinear_PDE})
without boundary conditions, see \cite{Olver1993}. Geometrically,
the condition that the Lie derivatives (\ref{eq:SymmetryGroup_Cond_PDE})
vanish on the submanifold $\mathcal{S}^{2}\subset J^{2}(\mathcal{E})$
means that the vector field $j^{2}(v)$ is tangent to $\mathcal{S}^{2}$.
Therefore, the corresponding transformation group $j^{2}(\Phi_{\varepsilon}):J^{2}(\mathcal{E})\rightarrow J^{2}(\mathcal{E})$
has the property 
\begin{equation}
j^{2}(\Phi_{\varepsilon})(\mathcal{S}^{2})\subset\mathcal{S}^{2}\,,\label{eq:j2PeS2_subset_S2}
\end{equation}
i.e. it maps all points of $\mathcal{S}^{2}$ again on $\mathcal{S}^{2}$.
As already remarked before, a section $\gamma:\bar{\Omega}\times\mathbb{R}^{+}\mathcal{\rightarrow\mathcal{E}}$
of the bundle $(\mathcal{E},\pi,\bar{\Omega}\times\mathbb{R}^{+})$
is a solution of (\ref{eq:Nonlinear_PDE}) if and only if the image
of $\Omega\times\mathbb{R}^{+}$ under the prolonged section $j^{2}(\gamma)$
lies in $\mathcal{S}^{2}$. If $j^{2}(\gamma)(\Omega\times\mathbb{R}^{+})$
lies in $\mathcal{S}^{2}$, then because of (\ref{eq:j2PeS2_subset_S2})
also 
\[
j^{2}(\Phi_{\varepsilon}\circ\gamma)(\Omega\times\mathbb{R}^{+})=j^{2}(\Phi_{\varepsilon})\circ j^{2}(\gamma)(\Omega\times\mathbb{R}^{+})
\]
lies in $\mathcal{S}^{2}$. Thus, the deformed section $\Phi_{\varepsilon}\circ\gamma$
is also a solution of (\ref{eq:Nonlinear_PDE}).

The additional condition (\ref{eq:SymmetryGroup_Cond_BC}) ensures
that the new solution $\Phi_{\varepsilon}\circ\gamma$ also satisfies
the boundary conditions (\ref{eq:Nonlinear_PDE_BC}). The proof relies
on the same arguments as before. First, it should be noted that the
restrictions $\left.j^{1}(v)\right|_{z=0}$ and $\left.j^{1}(v)\right|_{z=1}$
of the vector field $j^{1}(v)$ are vector fields on the manifolds
$\mathcal{B}_{A}$ and $\mathcal{B}_{B}$. Geometrically, the condition
(\ref{eq:SymmetryGroup_Cond_BC}) means that $\left.j^{1}(v)\right|_{z=0}$
and $\left.j^{1}(v)\right|_{z=1}$ are tangent to the submanifolds
$\mathcal{S}_{A}^{1}\subset\mathcal{B}_{A}$ and $\mathcal{S}_{B}^{1}\subset\mathcal{B}_{B}$
determined by the boundary conditions (\ref{eq:Nonlinear_PDE_BC}).
Therefore, the transformation group generated by $v$ maps solutions
that satisfy the boundary conditions again on solutions that satisfy
the boundary conditions. 
\end{proof}
\begin{rem}
It should be noted that the conditions of Theorem \ref{thm:SymmetryGroup}
are only sufficient conditions. They are not necessary, since we do
not make the assumption of local solvability. For the case without
boundary conditions, which is discussed in \cite{Olver1993}, local
solvability means, roughly speaking, that through every point of the
submanifold $\mathcal{S}^{2}\subset J^{2}(\mathcal{E})$ there passes
a solution of the PDEs (\ref{eq:Nonlinear_PDE}). With this assumption,
the conditions (\ref{eq:SymmetryGroup_Cond_PDE}) become necessary
and sufficient. 
\end{rem}

\section{\label{sec:ObservabilityAnalysis}Application to the Observability
Analysis}

With respect to the observability problem, we are interested in symmetry
groups that deform the solutions without changing the trajectories
of input and output. If we can find such a symmetry group $\Phi_{\varepsilon}$,
then we can transform every solution $\gamma$ into other solutions
$\Phi_{\varepsilon}\circ\gamma$ with different initial conditions
but the same input $u(t)$ and output $y(t)$. Thus, the initial condition
can never be determined uniquely from the input and the output.

To construct a symmetry group that does not change the input trajectory,
we simply have to set the component $v_{u}(z,t,x,u)$ of the infinitesimal
generator (\ref{eq:VerticalTransformationGroup_InfGen}) to zero,
i.e. we must consider vector fields of the form $v=v_{x}^{\alpha}(z,t,x,u)\partial_{x^{\alpha}}$.
The second requirement \textendash{} invariance of the output trajectory
\textendash{} means that the symmetry group must satisfy
\[
\left.c(t,x,x_{z})\circ j^{1}(\gamma)\right|_{z=z_{0}}=\left.c(t,x,x_{z})\circ j^{1}(\Phi_{\varepsilon}\circ\gamma)\right|_{z=z_{0}}
\]
for all solutions $\gamma$ and all $\varepsilon\in I_{0}$ in some
interval $I_{0}\subset\mathbb{R}$ containing zero, i.e. the output
must be the same for the solution $\gamma$ and all solutions $\Phi_{\varepsilon}\circ\gamma$
parametrized by the group parameter $\varepsilon$. Because of $j^{1}(\Phi_{\varepsilon}\circ\gamma)=j^{1}(\Phi_{\varepsilon})\circ j^{1}(\gamma)$
and the fact that $j^{1}(v)$ is the infinitesimal generator of $j^{1}(\Phi_{\varepsilon})$,
this condition holds if the Lie derivative 
\[
\left.L_{j^{1}(v)}c(t,x,x_{z})\right|_{z=z_{0}}
\]
vanishes at $z=z_{0}$. The following theorem summarizes our results.%
\begin{thm}
\label{thm:NonObservability}Consider the system (\ref{eq:Nonlinear_PDE})
with boundary conditions (\ref{eq:Nonlinear_PDE_BC}) and output (\ref{eq:Nonlinear_PDE_y}).
If there exists a smooth vector field 
\begin{equation}
v=v_{x}^{\alpha}(z,t,x,u)\partial_{x^{\alpha}}\label{eq:NonObservability_InfGen}
\end{equation}
on $\mathcal{E}$ with 
\begin{equation}
\left.\partial_{u}v_{x}^{\alpha}\right|_{t=0}=0\,,\quad\alpha=1,\ldots,n_{x}\label{eq:NonObservability_InfGen_InitialCond}
\end{equation}
and 
\begin{equation}
\left.v\right|_{t=0}\neq0\label{eq:NonObservability_InfGen_InitialCond_2}
\end{equation}
that satisfies the conditions
\begin{equation}
\begin{array}{rcc}
L_{j^{2}(v)}\left(x_{t}^{\alpha}-f^{\alpha}(z,t,x,x_{z},x_{zz},u)\right) & = & 0\end{array}\label{eq:NonObservability_DomainCond}
\end{equation}
on the submanifold $\mathcal{S}^{2}\subset J^{2}(\mathcal{E})$, the
conditions
\begin{equation}
\begin{array}{ccl}
\left.L_{j^{1}(v)}g^{\lambda}(t,x,x_{z})\right|_{z=0} & = & 0\\
\left.L_{j^{1}(v)}h^{\mu}(t,x,x_{z})\right|_{z=1} & = & 0
\end{array}\label{eq:NonObservability_BoundaryCond}
\end{equation}
on the submanifolds $\mathcal{S}_{A}^{1}\subset\mathcal{B}_{A}$ and
$\mathcal{S}_{B}^{1}\subset\mathcal{B}_{B}$, and the condition 
\begin{equation}
\left.L_{j^{1}(v)}c(t,x,x_{z})\right|_{z=z_{0}}=0\,,\label{eq:NonObservability_OutputCond}
\end{equation}
then the system is not observable. 
\end{thm}

\begin{proof}
Because of Theorem \ref{thm:SymmetryGroup}, a vector field (\ref{eq:NonObservability_InfGen})
that meets (\ref{eq:NonObservability_DomainCond}) and (\ref{eq:NonObservability_BoundaryCond})
generates a vertical 1-parameter symmetry group $\Phi_{\varepsilon}$
of the system (\ref{eq:Nonlinear_PDE}) with boundary conditions (\ref{eq:Nonlinear_PDE_BC}).
Since the vector field (\ref{eq:NonObservability_InfGen}) has no
components in $\partial_{u}$-direction, the conditions $L_{j^{2}(v)}u_{z}=0$,
$L_{j^{2}(v)}u_{zz}=0$, and $L_{j^{2}(v)}u_{zt}=0$ of Theorem \ref{thm:SymmetryGroup}
are always satisfied and do not need to be checked. For the same reason,
the symmetry group does not deform the trajectory of the input, and
because of condition (\ref{eq:NonObservability_OutputCond}) it does
not deform the trajectory of the output.

The condition (\ref{eq:NonObservability_InfGen_InitialCond}) means
that for $t=0$ the vector field is independent of the input $u$.
This ensures that the initial condition of a transformed solution
$\Phi_{\varepsilon}\circ\gamma$ depends only on the initial condition
$\gamma_{x}(z,0)$ of the original solution $\gamma=(z,t,\gamma_{x},\gamma_{u})$,
and not on the input $\gamma_{u}$ at time $t=0$.\footnote{For this reason, we often use the sloppy but convenient notation $\Phi_{x,\varepsilon}\circ\gamma_{x}(z,0)$
to express transformed initial conditions, even though $\Phi_{x,\varepsilon}(z,t,x,u)$
has of course more arguments than the variables $x$ that are determined
by $\gamma_{x}(z,0)$. Note also that we write $\gamma=(z,t,\gamma_{x},\gamma_{u})$
instead of just $\gamma=(\gamma_{x},\gamma_{u})$, since we defined
solutions geometrically as sections of the bundle $(\mathcal{E},\pi,\bar{\Omega}\times\mathbb{R}^{+})$,
and $\mathcal{E}$ has coordinates $(z,t,x,u)$.}

If we consider now two initial conditions $\gamma_{x}(z,0)$ and
\[
\bar{\gamma}_{x}(z,0)=\Phi_{x,\varepsilon}\circ\gamma_{x}(z,0)\,,
\]
where $\bar{\gamma}_{x}(z,0)$ is generated from $\gamma_{x}(z,0)$
by means of $\Phi_{\varepsilon}$ with some suitable value of the
group parameter $\varepsilon$, then they are clearly indistinguishable:
For every solution $\gamma=(z,t,\gamma_{x},\gamma_{u})$ with initial
condition $\gamma_{x}(z,0)$, because of the properties of the symmetry
group $\Phi_{\varepsilon}$ there exists a solution
\[
\begin{array}{cccccc}
\bar{\gamma} & = & \Phi_{\varepsilon}\circ\gamma & = & (z,t,\underbrace{\Phi_{x,\varepsilon}\circ\gamma}\,, & \underbrace{\gamma_{u}})\\
 &  &  &  & \bar{\gamma}_{x} & \bar{\gamma}_{u}
\end{array}
\]
with initial condition $\bar{\gamma}_{x}(z,0)$ that has the same
input $\bar{\gamma}_{u}=\gamma_{u}$ and yields the same output
\[
\left.c(t,x,x_{z})\circ j^{1}(\bar{\gamma})\right|_{z=z_{0}}=\left.c(t,x,x_{z})\circ j^{1}(\gamma)\right|_{z=z_{0}}\,.
\]
Consequently, for every initial condition $\gamma_{x}(z,0)$, the
symmetry group generates a set of indistinguishable initial conditions
\[
\mathcal{I}(\gamma_{x}(z,0))=\left\{ \left.\Phi_{x,\varepsilon}\circ\gamma_{x}(z,0)\right|\varepsilon\in I_{0}\right\} 
\]
by varying the group parameter $\varepsilon$ in some interval $I_{0}\subset\mathbb{R}$
containing zero. Since condition (\ref{eq:NonObservability_InfGen_InitialCond_2})
guarantees that the vector field (\ref{eq:NonObservability_InfGen})
does not vanish for $t=0$,\footnote{We have included the condition (\ref{eq:NonObservability_InfGen_InitialCond_2})
only for the sake of completeness. If the vector field (\ref{eq:NonObservability_InfGen})
would vanish for $t=0$, the corresponding symmetry group would generate
different solutions with the same initial condition and the same input
trajectory. However, since we have assumed that the solution is uniquely
determined by the initial condition and the input (see Section \ref{sec:GeometricRepresentation}),
this cannot happen.} for every initial condition $\gamma_{x}(z,0)$ the set of indistinguishable
initial conditions $\mathcal{I}(\gamma_{x}(z,0))$ contains more than
one element (in fact, infinitely many elements parametrized by $\varepsilon$),
and therefore the system is not observable. 
\end{proof}

It should be noted that, even though the original system is nonlinear,
the obtained conditions (\ref{eq:NonObservability_DomainCond}), (\ref{eq:NonObservability_BoundaryCond}),
and (\ref{eq:NonObservability_OutputCond}) are linear PDEs in the
unknown coefficients $v_{x}^{\alpha}$ of the vector field (\ref{eq:NonObservability_InfGen}).
In the following, we demonstrate the approach by means of two examples.

\subsection{A Simple Nonlinear Example}

Consider the system
\begin{equation}
\begin{array}{ccl}
\partial_{t}x^{1}(z,t) & = & x^{2}(z,t)\\
\partial_{t}x^{2}(z,t) & = & \partial_{z}^{2}x^{1}(z,t)-x^{2}(z,t)^{3}+u(t)
\end{array}\label{eq:NonlinEx_PDE}
\end{equation}
on the domain $\Omega=(0,1)$ with Neumann boundary conditions
\begin{equation}
\begin{array}{ccc}
\partial_{z}x^{1}(0,t) & = & 0\\
\partial_{z}x^{1}(1,t) & = & 0
\end{array}\label{eq:NonlinEx_BC}
\end{equation}
at $z=0$ and $z=1$, and the output
\begin{equation}
y=x^{2}(0,t)\label{eq:NonlinEx_Output}
\end{equation}
at $z_{0}=0$. This system is a wave equation with a nonlinear damping
described by the term $-x^{2}(z,t)^{3}$, and reflecting boundary
conditions at both ends. The input $u(t)$ can be interpreted e.g.
as an equally distributed force density, and the output is the velocity
$x^{2}$ at the left end $z_{0}=0$.

It can be verified easily that the vector field
\begin{equation}
v=\partial_{x^{1}}\label{eq:NonlinEx_InfGen}
\end{equation}
satisfies the conditions (\ref{eq:NonObservability_DomainCond}),
(\ref{eq:NonObservability_BoundaryCond}), and (\ref{eq:NonObservability_OutputCond})
of Theorem \ref{thm:NonObservability}. Because of
\[
j^{1}(v)=j^{2}(v)=\partial_{x^{1}}\,,
\]
we have
\[
\begin{array}{rcc}
L_{j^{2}(v)}\left(x_{t}^{1}-x^{2}\right) & = & 0\\
L_{j^{2}(v)}\left(x_{t}^{2}-x_{zz}^{1}+(x^{2})^{3}-u\right) & = & 0
\end{array}
\]
(even on $J^{2}(\mathcal{E})$ and not only on $\mathcal{S}^{2}\subset J^{2}(\mathcal{E})$)
as well as
\[
\begin{array}{ccc}
\left.L_{j^{1}(v)}x_{z}^{1}\right|_{z=0} & = & 0\\
\left.L_{j^{1}(v)}x_{z}^{1}\right|_{z=1} & = & 0
\end{array}
\]
(even on $\mathcal{B}_{A}$ and $\mathcal{B}_{B}$, and not only on
$\mathcal{S}_{A}^{1}\subset\mathcal{B}_{A}$ and $\mathcal{S}_{B}^{1}\subset\mathcal{B}_{B}$)
and
\[
\left.L_{j^{1}(v)}x^{2}\right|_{z=0}=0\,.
\]
The other conditions (\ref{eq:NonObservability_InfGen_InitialCond})
and (\ref{eq:NonObservability_InfGen_InitialCond_2}) are obviously
also satisfied. Consequently, according to Theorem \ref{thm:NonObservability}
the system is not observable. The vector field (\ref{eq:NonlinEx_InfGen})
generates the symmetry group
\[
\Phi_{\varepsilon}:(z,t,x^{1},x^{2},u)\rightarrow(z,t,x^{1}+\varepsilon,x^{2},u)\,,
\]
which simply adds a constant offset to the value of $x^{1}$. Thus,
it maps a solution
\begin{equation}
\gamma=(z,t,\gamma_{x}^{1},\gamma_{x}^{2},\gamma_{u})\label{eq:NonlinEx_Sol}
\end{equation}
to the solution
\begin{equation}
\Phi_{\varepsilon}\circ\gamma=(z,t,\gamma_{x}^{1}+\varepsilon,\gamma_{x}^{2},\gamma_{u})\,,\label{eq:NonlinEx_TransSol}
\end{equation}
and the corresponding initial conditions $(\gamma_{x}^{1}(z,0),\gamma_{x}^{2}(z,0))$
and $(\gamma_{x}^{1}(z,0)+\varepsilon,\gamma_{x}^{2}(z,0))$ are indistinguishable.

In this example it is of course obvious that (\ref{eq:NonlinEx_TransSol})
is again a solution and produces the same output trajectory as (\ref{eq:NonlinEx_Sol}),
since in the PDEs (\ref{eq:NonlinEx_PDE}), the boundary conditions
(\ref{eq:NonlinEx_BC}), and the output (\ref{eq:NonlinEx_Output})
there appear only derivatives of $x^{1}$, but not its absolute value.

\subsection{An Academic Example}

Consider the system
\begin{equation}
\begin{array}{ccl}
\partial_{t}x^{1}(z,t) & = & x^{1}(z,t)x^{2}(z,t)\partial_{z}^{2}x^{1}(z,t)+u(t)\\
\partial_{t}x^{2}(z,t) & = & \partial_{z}x^{2}(z,t)-x^{2}(z,t)^{2}\partial_{z}^{2}x^{1}(z,t)+\\
 &  & +\tfrac{x^{2}(z,t)}{x^{1}(z,t)}\left(\partial_{z}x^{1}(z,t)-u(t)\right)
\end{array}\label{eq:AcademicEx_PDE}
\end{equation}
on the domain $\Omega=(0,1)$ with the boundary conditions
\begin{equation}
\begin{array}{rcc}
\partial_{z}x^{1}(0,t) & = & 0\\
x^{1}(0,t)x^{2}(0,t)-1 & = & 0
\end{array}\label{eq:AcademicEx_BC_0}
\end{equation}
at $z=0$ and
\begin{equation}
\begin{array}{ccc}
\partial_{z}x^{1}(1,t) & = & 0\end{array}\label{eq:AcademicEx_BC_1}
\end{equation}
at $z=1$. The output is
\begin{equation}
y=\partial_{z}x^{1}(\tfrac{1}{2},t)\label{eq:AcademicEx_Output}
\end{equation}
at $z_{0}=\tfrac{1}{2}$. It can be verified that the vector field
\begin{equation}
v=\partial_{x^{1}}-\tfrac{x^{2}}{x^{1}}\partial_{x^{2}}\label{eq:AcademicEx_InfGen}
\end{equation}
satisfies all conditions of Theorem \ref{thm:NonObservability}. First,
since the vector field (\ref{eq:AcademicEx_InfGen}) depends neither
on $t$ nor on $u$, the conditions (\ref{eq:NonObservability_InfGen_InitialCond})
and (\ref{eq:NonObservability_InfGen_InitialCond_2}) are certainly
met. The expressions for the prolongations $j^{1}(v)$ and $j^{2}(v)$
are too extensive to be presented here, but it is not hard to verify
with a computer algebra system that the Lie derivatives
\[
\begin{array}{r}
L_{j^{2}(v)}\left(x_{t}^{1}-x^{1}x^{2}x_{zz}^{1}-u\right)\\
L_{j^{2}(v)}\left(x_{t}^{2}-x_{z}^{2}+(x^{2})^{2}x_{zz}^{1}-\tfrac{x^{2}}{x^{1}}\left(x_{z}^{1}-u\right)\right)
\end{array}
\]
vanish on $\mathcal{S}^{2}\subset J^{2}(\mathcal{E})$, the Lie derivatives
\[
\begin{array}{l}
\left.L_{j^{1}(v)}x_{z}^{1}\right|_{z=0}\\
\left.L_{j^{1}(v)}(x^{1}x^{2}-1)\right|_{z=0}
\end{array}
\]
vanish on $\mathcal{B}_{A}$ (and not only on $\mathcal{S}_{A}^{1}\subset\mathcal{B}_{A}$),
and the Lie derivative
\[
\begin{array}{c}
\left.L_{j^{1}(v)}x_{z}^{1}\right|_{z=1}\end{array}
\]
vanishes on $\mathcal{B}_{B}$ (and not only on $\mathcal{S}_{B}^{1}\subset\mathcal{B}_{B}$).
Since the Lie derivative
\[
\left.L_{j^{1}(v)}x_{z}^{1}\right|_{z=\tfrac{1}{2}}
\]
of the output also vanishes, all conditions are met and the system
is not observable.

\section{\label{sec:LinearSystems}Linear Systems}

In this section, we shall discuss how the results of Section \ref{sec:ObservabilityAnalysis}
simplify for linear systems. In the nonlinear case, the conditions
of Theorem \ref{thm:NonObservability} are sufficient for non-observability.
For linear systems, we show that they become necessary and sufficient.
In other words, a linear system is not observable if and only if there
exists a symmetry group that does not change the trajectories of the
input and the output. We also show that the symmetry group approach
is closely related to the notion of non-observable subspace from infinite-dimensional
linear systems theory. Here it is important to remark that for infinite-dimensional
linear systems there exist the concepts of approximate and exact observability,
see \cite{CurtainZwart1995}. %
These concepts are defined in terms of the observability map, which
assigns to each initial condition the corresponding output trajectory
that is generated by the homogenous system without input. Approximate
observability means that the observability map is injective, whereas
exact observability requires in addition that the inverse of the observability
map is also bounded, and therefore continuous. Our definition of observability
via the non-existence of indistinguishable initial conditions corresponds
to the injectivity of the observability map, and therefore to approximate
observability in the sense of \cite{CurtainZwart1995}.

In the following, we consider linear time-invariant PDEs of the form
\begin{equation}
\begin{array}{ccc}
\partial_{t}x^{\alpha}(z,t) & = & A_{\beta}^{\alpha}(z)x^{\beta}(z,t)+A_{z,\beta}^{\alpha}(z)\partial_{z}x^{\beta}(z,t)+\\
 &  & +A_{zz,\beta}^{\alpha}(z)\partial_{z}^{2}x^{\beta}(z,t)+B^{\alpha}(z)u(t)\,,
\end{array}\label{eq:Linear_PDE}
\end{equation}
$\alpha=1,\ldots,n_{x}$, on a 1-dimensional spatial domain $\Omega=(0,1)\subset\mathbb{R}$
with boundary conditions 
\begin{equation}
\begin{array}{ccc}
G_{\beta}^{\lambda}x^{\beta}(0,t)+G_{z,\beta}^{\lambda}\partial_{z}x^{\beta}(0,t) & = & 0\,,\quad\lambda=1,\ldots,n_{A}\\
H_{\beta}^{\mu}x^{\beta}(1,t)+H_{z,\beta}^{\mu}\partial_{z}x^{\beta}(1,t) & = & 0\,,\quad\mu=1,\ldots,n_{B}
\end{array}\label{eq:Linear_PDE_BC}
\end{equation}
and an output function 
\begin{equation}
y(t)=C_{\beta}x^{\beta}(z_{0},t)+C_{z,\beta}\partial_{z}x^{\beta}(z_{0},t)\label{eq:Linear_PDE_Output}
\end{equation}
defined at some point $z_{0}\in\bar{\Omega}$. This system class is
a special case of the nonlinear systems considered in the previous
sections, and includes e.g. the heat equation and the wave equation,
with homogenous Dirichlet, Neumann, and Robin boundary conditions.
Geometrically, the PDEs (\ref{eq:Linear_PDE}) can be represented
as a submanifold $\mathcal{S}^{2}\subset J^{2}(\mathcal{E})$ described
by the equations 
\[
\begin{array}{rcl}
x_{t}^{\alpha}-A_{\beta}^{\alpha}(z)x^{\beta}-A_{z,\beta}^{\alpha}(z)x_{z}^{\beta}-\\
-A_{zz,\beta}^{\alpha}(z)x_{zz}^{\beta}-B^{\alpha}(z)u & = & 0\,,\quad\alpha=1,\ldots,n_{x}\\
u_{z} & = & 0\\
u_{zz} & = & 0\\
u_{zt} & = & 0\,.
\end{array}
\]
The additional equations for $u_{z}$, $u_{zz}$, and $u_{zt}$ again
incorporate that we are only interested in solutions where $u$ does
not depend on $z$.%
{} The boundary conditions (\ref{eq:Linear_PDE_BC}) are again equations
on manifolds $\mathcal{B}_{A}$ and $\mathcal{B}_{B}$ with coordinates
$(t,x,u,x_{z},x_{t},u_{z},u_{t})$, which describe submanifolds $\mathcal{S}_{A}^{1}\subset\mathcal{B}_{A}$
and $\mathcal{S}_{B}^{1}\subset\mathcal{B}_{B}$.

A fundamental difference to the nonlinear case is that for linear
systems it is sufficient to consider symmetry groups with infinitesimal
generators of the form
\begin{equation}
v=v_{x}^{\alpha}(z,t)\partial_{x^{\alpha}}+v_{u}(z,t)\partial_{u}\,,\label{eq:InfGen_linear_with_vu}
\end{equation}
where the coefficients only depend on the independent variables $z$
and $t$. This can be justified by the superposition principle: Because
of the superposition principle, for every pair of solutions $\gamma$
and $\bar{\gamma}$ we can construct a symmetry group 
\begin{multline}
\Phi_{\varepsilon}:(z,t,x,u)\rightarrow(z,t,x+(\bar{\gamma}_{x}(z,t)-\gamma_{x}(z,t))\varepsilon,\\
u+(\bar{\gamma}_{u}(z,t)-\gamma_{u}(z,t))\varepsilon)\label{eq:SymmetryGroup_Superposition}
\end{multline}
with infinitesimal generator 
\begin{equation}
v=(\bar{\gamma}_{x}^{\alpha}(z,t)-\gamma_{x}^{\alpha}(z,t))\partial_{x^{\alpha}}+(\bar{\gamma}_{u}(z,t)-\gamma_{u}(z,t))\partial_{u}\,,\label{eq:InfGen_Superposition}
\end{equation}
that deforms $\gamma$ for $\varepsilon=1$ into $\bar{\gamma}$.
Since the vector field (\ref{eq:InfGen_Superposition}) is of the
form (\ref{eq:InfGen_linear_with_vu}), we can construct every solution
$\bar{\gamma}$ by deforming a given solution $\gamma$ with a symmetry
group of this special type. Thus, for our application there is no
advantage in considering the general case, where the coefficients
of the infinitesimal generator (\ref{eq:InfGen_linear_with_vu}) may
also depend on $x$ and $u$. 
\begin{rem}
Note that (\ref{eq:SymmetryGroup_Superposition}) is indeed a symmetry
group in compliance with Definition \ref{def:SymmetryGroup}. Because
of the superposition principle, it transforms all solutions into other
solutions, and not only the solutions $\gamma$ and $\bar{\gamma}$
that were used to construct it. 
\end{rem}

Now let us apply Theorem \ref{thm:NonObservability} to the linear
case. First, we can replace the vector field (\ref{eq:NonObservability_InfGen})
by the vector field
\begin{equation}
v=v_{x}^{\alpha}(z,t)\partial_{x^{\alpha}}\,.\label{eq:InfGen_linear}
\end{equation}
With the second prolongation 
\[
\begin{array}{ccl}
j^{2}(v) & = & v^{\alpha}\partial_{x^{\alpha}}+d_{z}(v^{\alpha})\partial_{x_{z}^{\alpha}}+d_{t}(v^{\alpha})\partial_{x_{t}^{\alpha}}+\\
 &  & +d_{zz}(v^{\alpha})\partial_{x_{zz}^{\alpha}}+d_{zt}(v^{\alpha})\partial_{x_{zt}^{\alpha}}+d_{tt}(v^{\alpha})\partial_{x_{tt}^{\alpha}}
\end{array}
\]
of (\ref{eq:InfGen_linear}), an evaluation of the condition (\ref{eq:NonObservability_DomainCond})
with 
\[
f^{\alpha}=A_{\beta}^{\alpha}(z)x^{\beta}+A_{z,\beta}^{\alpha}(z)x_{z}^{\beta}+A_{zz,\beta}^{\alpha}(z)x_{zz}^{\beta}+B^{\alpha}(z)u
\]
yields 
\[
d_{t}(v^{\alpha})-A_{\beta}^{\alpha}(z)v^{\beta}-A_{z,\beta}^{\alpha}(z)d_{z}(v^{\beta})-A_{zz,\beta}^{\alpha}(z)d_{zz}(v^{\beta})=0\,.
\]
Since the coefficients of $v$ only depend on $z$ and $t$, the total
derivatives degenerate to partial derivatives, and we obtain
\begin{equation}
\begin{array}{ccr}
\partial_{t}v^{\alpha}(z,t) & = & A_{\beta}^{\alpha}(z)v^{\beta}(z,t)+A_{z,\beta}^{\alpha}(z)\partial_{z}v^{\beta}(z,t)+\\
 &  & +A_{zz,\beta}^{\alpha}(z)\partial_{z}^{2}v^{\beta}(z,t)\,.
\end{array}\label{eq:NonObservability_DomainCond_linear}
\end{equation}
This is just the homogenous part of the original PDEs (\ref{eq:Linear_PDE}).
Note that in (\ref{eq:NonObservability_DomainCond_linear}) there
appear no variables $x_{t}^{\alpha}$, and therefore it makes no difference
whether we evaluate (\ref{eq:NonObservability_DomainCond_linear})
on the submanifold $\mathcal{S}^{2}\subset J^{2}(\mathcal{E})$ determined
by the system equations, or on $J^{2}(\mathcal{E})$ itself. If the
conditions hold on $\mathcal{S}^{2}$, then they also hold on $J^{2}(\mathcal{E})$.
Next, an evaluation of the condition (\ref{eq:NonObservability_BoundaryCond})
with 
\[
\begin{array}{ccc}
g^{\lambda} & = & G_{\beta}^{\lambda}x^{\beta}+G_{z,\beta}^{\lambda}x_{z}^{\beta}\\
h^{\mu} & = & H_{\beta}^{\mu}x^{\beta}+H_{z,\beta}^{\mu}x_{z}^{\beta}
\end{array}
\]
yields 
\[
\begin{array}{ccl}
\left.\left(G_{\beta}^{\lambda}v^{\beta}+G_{z,\beta}^{\lambda}d_{z}(v^{\beta})\right)\right|_{z=0} & = & 0\\
\left.\left(H_{\beta}^{\mu}v^{\beta}+H_{z,\beta}^{\mu}d_{z}(v^{\beta})\right)\right|_{z=1} & = & 0\,.
\end{array}
\]
For the same reason as above, the total derivatives degenerate to
partial derivatives, and we obtain
\begin{equation}
\begin{array}{ccl}
G_{\beta}^{\lambda}v^{\beta}(0,t)+G_{z,\beta}^{\lambda}\partial_{z}v^{\beta}(0,t) & = & 0\\
H_{\beta}^{\mu}v^{\beta}(1,t)+H_{z,\beta}^{\mu}\partial_{z}v^{\beta}(1,t) & = & 0\,.
\end{array}\label{eq:NonObservability_BoundaryCond_linear}
\end{equation}
This are just the original boundary conditions (\ref{eq:Linear_PDE_BC}).
Since in (\ref{eq:NonObservability_BoundaryCond_linear}) there occur
no variables $x^{\alpha}$ or $x_{z}^{\alpha}$, it makes again no
difference whether we evaluate (\ref{eq:NonObservability_BoundaryCond_linear})
on the submanifolds $\mathcal{S}_{A}^{1}\subset\mathcal{B}_{A}$ and
$\mathcal{S}_{B}^{1}\subset\mathcal{B}_{B}$ determined by the boundary
conditions, or on $\mathcal{B}_{A}$ and $\mathcal{B}_{B}$ themselves.
Thus, we can already observe that the coefficients of the infinitesimal
generator (\ref{eq:InfGen_linear}) of a vertical symmetry group must
satisfy the homogenous part of the original PDEs (\ref{eq:Linear_PDE})
with the original boundary conditions (\ref{eq:Linear_PDE_BC}). Finally,
an evaluation of the condition (\ref{eq:NonObservability_OutputCond})
with 
\[
c=C_{\beta}x^{\beta}+C_{z,\beta}x_{z}^{\beta}
\]
yields
\[
\left.\left(C_{\beta}v^{\beta}+C_{z,\beta}d_{z}(v^{\beta})\right)\right|_{z=z_{0}}=0\,,
\]
and if we replace again the total derivatives by partial derivatives
we get 
\begin{equation}
C_{\beta}v^{\beta}(z_{0},t)+C_{z,\beta}\partial_{z}v^{\beta}(z_{0},t)=0\,.\label{eq:NonObservability_OutputCond_linear}
\end{equation}
The left-hand side of (\ref{eq:NonObservability_OutputCond_linear})
is just the original system output (\ref{eq:Linear_PDE_Output}).
Thus, for the coefficients of a vector field (\ref{eq:InfGen_linear})
that satisfies the conditions of Theorem \ref{thm:NonObservability}
we need (non-trivial) solutions of the homogenous part of the PDEs
(\ref{eq:Linear_PDE}) with the boundary conditions (\ref{eq:Linear_PDE_BC})
that generate an output (\ref{eq:Linear_PDE_Output}) which is identically
zero, i.e. $y(t)=0$ for all $t$.

However, it is well-known from infinite-dimensional linear systems
theory that non-trivial solutions that generate an output which is
identically zero exist if and only if the system is not (approximately)
observable, see e.g. \cite{CurtainZwart1995}: Because of the superposition
principle, for linear systems the output is the sum of a part that
depends on the initial condition and a part that depends on the input.
Therefore, the question whether there exists a choice of the input
such that two initial conditions $\gamma_{x}(z,0)$ and $\bar{\gamma}_{x}(z,0)$
produce different outputs reduces to the question whether they produce
different outputs for the homogenous system without input. If they
produce the same output, then again because of the superposition principle
the initial condition $\bar{\gamma}_{x}(z,0)-\gamma_{x}(z,0)$ generates
such an output which is identically zero. Consequently, for linear
systems, the conditions of Theorem \ref{thm:NonObservability}, and
therefore the existence of a symmetry group $\Phi_{\varepsilon}$
that does not change the trajectories of the input and the output,
are necessary and sufficient for non-observability.

The set of initial conditions that generate an output identically
zero is a subspace of the infinite-dimensional state space which is
called the non-observable subspace, see \cite{CurtainZwart1995}.
Thus, for linear systems, the symmetry groups $\Phi_{\varepsilon}$
transform solutions $\gamma$ into other solutions $\Phi_{\varepsilon}\circ\gamma$
just in such a way that the difference $\left(\Phi_{x,\varepsilon}\circ\gamma_{x}(z,0)\right)-\gamma_{x}(z,0)$
of the initial conditions is an element of the non-observable subspace.

Of course, it is important to remark that a comparison between the
results of our symmetry group approach and the observability concepts
of \cite{CurtainZwart1995} suffers from two problems. First, as already
pointed out, for the calculation of symmetry groups we need a differential-geometric
framework and consider like in \cite{Olver1993} only smooth solutions,
whereas the semigroup theory used in \cite{CurtainZwart1995} is based
on mild or generalized solutions. Second, we consider outputs which
are defined at a single point $z_{0}\in\bar{\Omega}$. In the geometric
framework this is perfectly possible, whereas in \cite{CurtainZwart1995}
point outputs are not considered since the corresponding output maps
are in general not bounded. However, it would be as well possible
to consider instead of (\ref{eq:Nonlinear_PDE_y}) distributed outputs
\[
y(t)=c(t,x(z,t),\partial_{z}x(z,t))
\]
that are defined on the whole spatial domain $\bar{\Omega}$. The
only difference is that in Theorem \ref{thm:NonObservability} the
Lie derivative (\ref{eq:NonObservability_OutputCond}) of the output
would have to vanish for all $z\in\bar{\Omega}$, and not only at
$z_{0}$. The linear counterpart of such a distributed output is
\[
y(t)=C_{\beta}(z)x^{\beta}(z,t)+C_{z,\beta}(z)\partial_{z}x^{\beta}(z,t)\,.
\]
For the case $C_{z,\beta}(z)=0$ without derivatives, such a distributed
output might correspond to a bounded map from the state space to the
output space\footnote{The boundedness depends of course on the chosen norms of the involved
infinite-dimensional vector spaces.}, and therefore fit into the system class considered in \cite{CurtainZwart1995}.

We also want to remark that the observability problem is often considered
for a finite time interval $[0,\tau]$ with some $\tau>0$. In infinite-dimensional
linear systems theory this is called observability on $[0,\tau]$.
Although we have presented all our results for the infinite time interval
$\mathbb{R}^{+}$, they remain valid if we replace the space-time
manifold $\bar{\Omega}\times\mathbb{R}^{+}$ by $\bar{\Omega}\times[0,\tau]$.

\subsection{A Linear Example}

Consider the linear wave equation
\begin{equation}
\begin{array}{ccl}
\partial_{t}x^{1}(z,t) & = & x^{2}(z,t)\\
\partial_{t}x^{2}(z,t) & = & \partial_{z}^{2}x^{1}(z,t)+u(t)
\end{array}\label{eq:LinEx_PDE}
\end{equation}
on the domain $\Omega=(0,1)$ with Dirichlet boundary conditions
\begin{equation}
\begin{array}{ccc}
x^{1}(0,t) & = & 0\\
x^{1}(1,t) & = & 0
\end{array}\label{eq:LinEx_BC}
\end{equation}
at $z=0$ and $z=1$, and the output
\begin{equation}
y=x^{1}(\tfrac{1}{2},t)\label{eq:LinEx_Output}
\end{equation}
at $z_{0}=\tfrac{1}{2}$. The input $u(t)$ can be interpreted as
an equally distributed force density.

It may be verified that
\begin{equation}
\begin{array}{ccl}
x^{1}(z,t) & = & \sin(2\pi z)\cos(2\pi t)\\
x^{2}(z,t) & = & -2\pi\sin(2\pi z)\sin(2\pi t)
\end{array}\label{eq:LinEx_sol}
\end{equation}
is a solution of the homogenous part
\begin{equation}
\begin{array}{ccl}
\partial_{t}x^{1}(z,t) & = & x^{2}(z,t)\\
\partial_{t}x^{2}(z,t) & = & \partial_{z}^{2}x^{1}(z,t)
\end{array}\label{eq:LinEx_Hom_PDE}
\end{equation}
of (\ref{eq:LinEx_PDE}) that fulfills the boundary conditions (\ref{eq:LinEx_BC})
and produces an output (\ref{eq:LinEx_Output}) which is identically
zero. Thus, the corresponding initial condition
\[
\begin{array}{ccl}
x^{1}(z,0) & = & \sin(2\pi z)\\
x^{2}(z,0) & = & 0
\end{array}
\]
is an element of the non-observable subspace, and the system is not
observable. The vector field
\[
v=\sin(2\pi z)\cos(2\pi t)\partial_{x^{1}}-2\pi\sin(2\pi z)\sin(2\pi t)\partial_{x^{2}}
\]
with coefficients from (\ref{eq:LinEx_sol}) satisfies the conditions
of Theorem \ref{thm:NonObservability}, and is the infinitesimal generator
of a symmetry group
\begin{multline*}
\Phi_{\varepsilon}:(z,t,x^{1},x^{2},u)\rightarrow(z,t,x^{1}+\sin(2\pi z)\cos(2\pi t)\varepsilon,\\
x^{2}-2\pi\sin(2\pi z)\sin(2\pi t)\varepsilon,u)
\end{multline*}
that does not change the trajectories of input and output.

\section{Conclusions}

In this paper, we have suggested to use symmetry groups that do not
change the trajectories of the input and the output for proving that
a nonlinear infinite-dimensional system is not observable. Based on
a differential-geometric system representation, we have derived conditions
for the existence of such special symmetry groups. Even though the
original system is described by nonlinear PDEs, these conditions are
linear PDEs with additional restrictions. For linear infinite-dimensional
systems, the derived conditions simplify considerably, and become
necessary and sufficient for non-observability. In fact, they coincide
with the existence of a non-trivial non-observable subspace.

\bibliographystyle{IEEEtran}
\bibliography{IEEEabrv,Literature_Bernd}

\end{document}